\documentclass{article}
    \usepackage[title,titletoc,toc]{appendix}
\usepackage{fullpage}
\usepackage{setspace}
\usepackage[utf8]{inputenc}

\usepackage{comment}
\excludecomment{confidential}

\usepackage{dsfont}
    \usepackage{graphicx}
    \usepackage{amsmath, amssymb, latexsym, amsthm, url}
    \usepackage{mathrsfs,color,fancyhdr}
    \usepackage{cite}
    \usepackage{prettyref}
      \newtheorem{thm}{Theorem}[section]
      \newcommand{\bthm}{\begin{thm}} \newcommand{\ethm}{\end{thm}}
      \newtheorem{prop}[thm]{Proposition}
      \newcommand{\bprp}{\begin{prop}} \newcommand{\eprp}{\end{prop}}
      \newtheorem{fact}[thm]{Fact}
      \newcommand{\bfct}{\begin{fact}} \newcommand{\efct}{\end{fact}}
      \newtheorem{prob}[thm]{Problem}
      \newcommand{\bprb}{\begin{prob}} \newcommand{\eprb}{\end{prob}}
      \newtheorem{lem}[thm]{Lemma}
      \newcommand{\blem}{\begin{lem}} \newcommand{\elem}{\end{lem}}
      \newtheorem{claim}[thm]{Claim}
      \newcommand{\bclm}{\begin{claim}} \newcommand{\eclm}{\end{claim}}
      \newtheorem{cor}[thm]{Corollary}
      \newcommand{\bcor}{\begin{cor}} \newcommand{\ecor}{\end{cor}}
      \newtheorem{conj}[thm]{Conjecture}
      \newcommand{\bcnj}{\begin{conj}} \newcommand{\ecnj}{\end{conj}}
      \theoremstyle{definition}
      \newtheorem{defn}[thm]{Definition}
      \newcommand{\bdfn}{\begin{defn}} \newcommand{\edfn}{\end{defn}}
      \theoremstyle{remark}
      \newtheorem{rem}[thm]{Remark}
      \newcommand{\brem}{\begin{rem}} \newcommand{\erem}{\end{rem}}
      \newtheorem{cnv}[thm]{Convention}
      \newcommand{\bcnv}{\begin{cnv}} \newcommand{\ecnv}{\end{cnv}}
      \newtheorem{exam}[thm]{Example}
      \newcommand{\bexm}{\begin{exam}} \newcommand{\eexm}{\end{exam}}
      \newcommand{\bpf}{\begin{proof}} \newcommand{\epf}{\end{proof}}
      \newtheorem{exer}[thm]{Exercise}
      \newcommand{\bexer}{\begin{exer}} \newcommand{\eexer}{\end{exer}}
      \newcommand{\ben}{\begin{enumerate}}
      \newcommand{\een}{\end{enumerate}}
      \newcommand{\bit}{\begin{itemize}}
      
      \newcommand{\eit}{\end{itemize}}

        \newcommand{\reals}{{\mathbb R}}

        \newcommand{\Z}{{\mathbb Z}}
        
        \newcommand{\N}{{\mathbb N}}
        \newcommand{\CC}{\mathcal{C}}

        \newcommand{\F}{\mathcal{F}}
        \newcommand{\pr}{\mathbb{P}}

        \newcommand{\E}{\mathbb{E}}
        \newcommand{\om}{\omega}
        \newcommand{\Om}{\Omega}

        \newcommand{\ignore}[1]{}
       
       \newcommand{\one}{\mathds{1}}
	\newcommand{\dif}{\text{d}}

  \newcommand{\intpart}[1]{\lfloor #1 \rfloor}
  \newcommand{\Embedding}{\iota}
\newcommand{\norm}[1]{\vert\vert #1 \vert\vert_{G^2(\mathbb{R}^d)}}

\usepackage{color}

\definecolor{Red}{rgb}{1,0,0}
\definecolor{Blue}{rgb}{0,0,1}
\definecolor{Olive}{rgb}{0.41,0.55,0.13}
\definecolor{Yarok}{rgb}{0,0.5,0}
\definecolor{Green}{rgb}{0,1,0}
\definecolor{MGreen}{rgb}{0,0.8,0}
\definecolor{DGreen}{rgb}{0,0.55,0}
\definecolor{Yellow}{rgb}{1,1,0}
\definecolor{Cyan}{rgb}{0,1,1}
\definecolor{Magenta}{rgb}{1,0,1}
\definecolor{Orange}{rgb}{1,.5,0}
\definecolor{Violet}{rgb}{.5,0,.5}
\definecolor{Purple}{rgb}{.75,0,.25}
\definecolor{Brown}{rgb}{.75,.5,.25}
\definecolor{Grey}{rgb}{.5,.5,.5}

    \date{\today}

    \title{Ballistic random walks in random environment as rough paths: convergence and area anomaly}
   \author{Olga Lopusanschi\footnote{Paris, {\tt Email: olga.lopusanschi@gmail.com}}
    \, and Tal Orenshtein\footnote{TU Berlin and WIAS Berlin, {\tt Email: orenshtein@wias-berlin.de}}}
\begin{document}
  \date{}
    \maketitle

\begin{abstract}
Annealed functional CLT in the rough path topology is proved for the standard class of ballistic random walks in random environment.
Moreover, the `area anomaly', i.e.\ a deterministic linear correction for the second level iterated integral of the rescaled path, is identified in terms
of a stochastic area on a regeneration interval.
The main theorem is formulated in more general settings, namely for any discrete process with uniformly bounded
increments which admits a regeneration structure where the regeneration times have finite moments.
Here the largest finite moment translates into the degree of regularity of the rough path topology.
In particular, the convergence holds in the $\alpha$-H\"older rough path topology for all $\alpha<1/2$ whenever all moments are finite, which is the case for the class of ballistic random walks in random environment. The latter may be compared to a special class of random walks in Dirichlet environments for which the regularity $\alpha<1/2$ is bounded away from $1/2$, explicitly in terms of the corresponding trap parameter.
\end{abstract}
\hfill


%%%%%%%%%%%%%%%%%%%%%%%%%%%%%%%%%%%%%%%%%%%%%%%%%%%%%%%%%%%%%%%%%5
\section{Introduction}\label{sec:intro}
%%%%%%%%%%%%%%%%%%%%%%%%%%%%%%%%%%%%%%%%%%%%%%%%%%%%%%%%%%%%%%%%%5
Rough path theory has been extensively developing since it was introduced by T.~Lyons in `98 \cite{lyons1998differential}.
The theory provides a framework to solutions to SDEs driven by non-regular signals
such as Brownian motions, while keeping the solution map continuous with respect to the signal.
The It\^{o} theory of stochastic integration,
being an $L^2$ theory in essence, does not allow integration path-by-path,
and hence does not give rise to solutions with such continuity property.

As it was observed by Lyons, the difficulty is not only a technical issue;
in any separable Banach space $\mathcal{B}\subset\CC[0,1]$ containing the sample paths of Brownian motions a.s.\
the map $(f,g)\to \int_{0}^\cdot f(t)\dot g(t)\dif t$ defined on smooth maps \emph{cannot} be extended to a continuous map
on $\mathcal{B}\times\mathcal{B}$ (see \cite[Proposition 1.1]{friz2014course} and the references therein).
Some additional information on the path is needed to achieve continuity, namely the so called ``iterated integrals'',
where the number of iterations needed is determined by the regularity of the signal.

Fix $T>0$ and $X:[0,T]\to \reals^d$. The $M$-th level iterated integral of $X$ is
\begin{equation}\label{eq:N iterated integral}
S^{M}_{s,t}(X)=\int_{s<u_1<...<u_M<t}\dif X_{u_1} \cdots \dif X_{u_M},\, s<t, s,t\in [0,T].
\end{equation}
Note that the definition of iterated integrals assumes a notion of integration with respect to $X$.

Lyons' theory uses the information coming from the iterated integral as a postulated high level
information and constructs a space (called the rough path space) in which solutions to SDEs driven by Brownian motion are continuous with respect to the latter.
In this case two levels of iteration are enough since the Brownian motion is $\alpha$-H\"older for some $\alpha>\frac{1}{3}$
(and actually for all $\alpha<\frac{1}{2}$). More generally, roughly speaking,
in case the signal is $\alpha$-H\"older continuous for some $\alpha\in (0,1]$, then $M=\lfloor 1/\alpha \rfloor$ levels of iteration are sufficient (and necessary).

For discrete processes with regeneration structure such as ballistic random walks in random environment (RWRE),
invariance principles are well known.
Our main result, Theorem \ref{thm:area anomaly in general conditions}, shows that after lifting the path we have as well a
scaling limit in the rough path topology where the regularity is determined by the moments of the regenerations.

The application to the so-called ballistic RWREs, formulated in Theorem \ref{thm:area anomaly anneald rwre}, is then immediate, and since regeneration times have all moments \cite{sznitman2000slowdown} the convergence in
spaces of regularity $\alpha$ is taken all the way to $\alpha <\frac{1}{2}$. The theorem is also applied to random walks in Dirichlet environments with large enough trap parameter, where in this case the convergence is on a limited regularity space, see Theorem \ref{thm:area anomaly anneald Dirichlet rwre} for the precise statement.

When a scaling limit is known for some process in the uniform topology, one might be interested to get a richer information about the limit. For inhomogeneous random walks with regeneration structure, an interesting phenomenon yields.
As it turns out, unlike the ``classical'' invariance principles,
when considering the second level iterated integral, which is related to the running signed area of the process as
we show, the local fluctuations do \emph{not} disappear in the limit,
and a correction has to be considered.
Moreover, thanks to the i.i.d\ structure of the walk on regeneration intervals,
the law of large numbers allows us to write the correction as
a linear function in time $(t\Gamma)_{0\le t\le T}$, called the area anomaly. In particular, $\Gamma$ is a deterministic matrix which is the expected signed area accumulated in a regeneration interval, divided by its expected length,
see the main result, Theorem \ref{thm:area anomaly in general conditions}.

Another application is related to the Wong-Zakai type approximations of solutions to SDEs.
Let $(B^N)_N$ be a sequence of semimartingales converging weakly in the uniform topology to a Brownian motion $B$.
An interesting question is to understand the approximating differential equations, where the noise is replaces by $B^N$.
Let $X$ be a solution to a SDE with nice (in an appropriate sense) drift and diffusion coefficients and
let $X^N$ be a solution to corresponding difference equation driven by $B^{N}$.
The Wong-Zakai Theorem implies that it is not true in general that $X^N$ converges to $X$ whenever the convergence of the noise holds in the uniform topology \cite{wong1965convergence}.
However, if the weak convergence of $B^N$ to $B$ holds in the rough path space of regularity $\alpha$ with a linear area correction $t\Gamma$, for some $\alpha\in(\frac{1}{3},\frac{1}{2})$, then the answer is affirmative, where the SDE under consideration has to be modified by adding a drift term which is explicit in terms of $\Gamma$ \cite{kelly2016rough}.

Other aspect of noise approximations effects is related to SPDEs. The theory of rough path was strengthened with Gubinelli's notion of controlled rough path \cite{gubinelli2004controlling} and branched rough path \cite{gubinelli2010ramification} which extend the notion of integration and of solutions to differential equation with respect to an abstract data coming from the noise.
This then inspired Martin Hairer to develop the far-reaching theory of regularity structures \cite{hairer2014theory}, which is now extensively studied. A similar question is fundamental to SPDEs: what can
we learn on the solutions if rather than mollifying the noise by a smooth function, one takes more complicated approximations? For a recent progress in this direction, see \cite{bruned2019algebraic}.

Going back to the Brownian case, the fundamental result related to our work is the Donsker's invariance principle in the rough path topology \cite{breuillard2009random}. An extension to random walks with general covariances was proved in \cite{kelly2016rough}.

In \cite{lopusanschi2017area} and \cite{lopusanschi2017levy} the authors studied some discrete
processes converging to Brownian motion in $\mathbb{R}^d$ in the rough path topology with area anomaly which was constructed explicitly. Our main idea of our proof is inspired by theirs, with two main differences. First, we do not use the strong Markov property for the excursions, which, for a finitely supported
jump distribution implies that the excursions have exponential tail. Instead, we only assume i.i.d.\
   regeneration structure and moments of the regeneration times.
Second, the discrete processes in these papers are homogeneous in space (a simple random walk on periodic graphs
  \cite{lopusanschi2017area}, or hidden Markov walk where the jumps are independent of the current location \cite{lopusanschi2017levy}). In our case we allow the process to have jump distributions that are inhomogeneous in space.

  Another interesting example is a Brownian motion in magnetic field. Here the discretization converges to an enhanced Brownian motion with an explicit area anomaly.

  The problem of discrete processes seen as rough paths is dealt with in other contexts as well.
  \cite{kelly2016rough} and \cite{kelly2016smooth}, and the more recent \cite{friz2018differential} used the
  rough path framework to deal with discrete approximations of SDEs. The case of random walks on nilpotent covering graphs was considered in \cite{ishiwata2018central,ishiwata2018centralparttwo,namba2018remark}
where the corresponding area anomaly is identified in terms of harmonic embeddings (see \cite[equation (2.6)]{ishiwata2020central}).
Anther paper concerning discrete processes which is of relevance here is \cite{chevyrev2019canonical}. In that paper the
authors showed a general construction of rough SDEs allowing rough paths with jumps beyond linear interpolations.

%\cite{chevyrev2018random}

%%%%%%%%%%%%%%%%%%%%%%%%%%%%%%%%%%%%%%%%%%%%%%%%%%%%%%%%%%%%%%%%%5
\subsection{Structure of the paper}\label{sec:struture}
%%%%%%%%%%%%%%%%%%%%%%%%%%%%%%%%%%%%%%%%%%%%%%%%%%%%%%%%%%%%%%%%%5
In order to keep the paper as self-contained as possible in Chapter \ref{sec:basics} we discuss
basic notions in rough path theory and set up the framework to be used in the rest of the paper.
In chapter \ref{sec:results} we formulate our main result,
Theorem \ref{thm:area anomaly in general conditions}. In chapter \ref{sec:examples} we give
some simple examples of processes converging in the rough path topology and lacking or having non-zero area anomaly. In Chapter \ref{applications} we present other special cases of our main result.
Particular cases are ballistic random walks in random environment, for which we also present an open problem, and random walks in Dirichlet environments. Finally, in Chapter \ref{sec:proof} we give the proof of Theorem \ref{thm:area anomaly in general conditions}.

%%%%%%%%%%%%%%%%%%%%%%%%%%%%%%%%%%%%%%%%%%%%%%%%%%%%%%%%%%%%%%%%%5
\section{Basic notions in rough path theory}\label{sec:basics}
%%%%%%%%%%%%%%%%%%%%%%%%%%%%%%%%%%%%%%%%%%%%%%%%%%%%%%%%%%%%%%%%%5
The aim of this section is to introduce briefly the basic objects in our framework. These are adapted from Chapters 2 and 3 of \cite{friz2014course}. The experienced reader can safely jump to Remark \ref{rem:why do we use G2}.
Since we assume here no familiarity of the reader with rough path theory we added a short discussion after Proposition \ref{prop: equiv of rough path norms} which is somewhat loosely formulated and should be treated accordingly. For an extensive account of the theory the reader is suggested to consult \cite{friz2010multidimensional} and \cite{friz2014course}.

Initially developed for solving differential equations, rough path theory is also useful in the discrete setting, and in particular for studying the convergence of discrete processes. For example,
in the uniform topology a simple random walk (SRW) on $\Z^2$ to which we add deterministic four steps clockwise loops every two steps
(see Figure \ref{fig:SRWLoop} below) converges to the same Brownian motion as a SRW which stays still for four steps every two steps.
Thus in the uniform topology the loops simply disappear at the limit.

\begin{figure}[!ht]
\begin{center}
\includegraphics[scale=0.4]{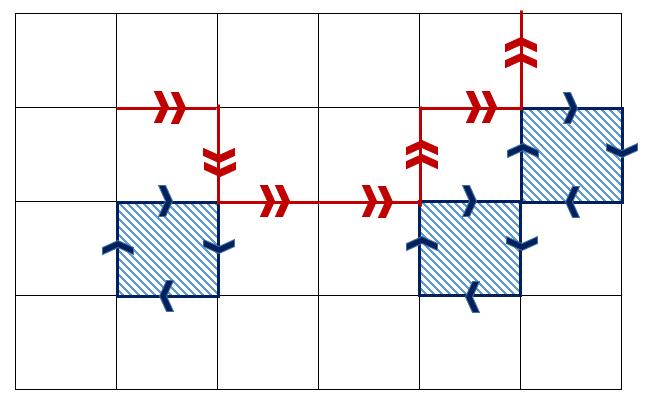}
\caption{A simple random walk with a deterministic loop every two steps. The double arrows (in red)
are the random walk's steps while the added loops are presented by the single arrows (in blue).}
   \label{fig:SRWLoop}
\end{center}
\end{figure}

The loops certainly do not play a role if one is interested in the limit trajectory only. However, if one wishes to study more aspects of the limit, the  accumulated area created by the loops could be also taken into consideration. A basic example for accumulated area in the continuous setting is provided by the ``bubble areas" of Lejay \cite{lejay2003introduction}. This weakness of the uniform topology is precisely one of the problems that rough path theory palliates.

%\emph{Notations.}
Following \cite{friz2010multidimensional} we denote by $\otimes$ two different actions:
   \begin{itemize}
     \item[•] for two elements of vector spaces, it is the usual tensor product: if $V$ and $W$ are $d$-dimensional, respectively $d'$-dimensional vector
     spaces, for $v\in V$ and $w\in W$, $v\otimes w$ is the matrix $(v_iw_j)_{1\leq i\leq d,1\leq j\leq d'}$;
     \item[•] for two elements of a group (in our case, $G^2(\mathbb{R}^d)$, defined below), it denotes the corresponding group operation.
   \end{itemize}

The continuous process obtained by linear interpolation (or any other piecewise $C^1$ interpolation) of a discrete process, as well as its iterated integrals can be encoded in terms of elements of a particular nilpotent Lie group (see \cite[Section 2.3]{friz2014course} for more details).
For simplicity, and since our motivation in this paper is to prove convergence to Brownian motion, which is
$\alpha$-H\"oldar for all $\alpha<1/2$, we adapt the
general point of view taken in the book \cite{friz2014course} and consider \eqref{eq:N iterated integral}
in the case $M\le 2$, i.e.\ with only two levels of iteration. Therefore in the rest of the paper, we write $S_{s,t}(X)$ for $S_{s,t}^2(X)$. The pairs $(X_{s,t},S_{s,t}(X))$, $s<t$, with $X_{s,t}=X_t-X_s$, for a smooth path $X$, have a natural group structure with respect to increment concatenation. Here is the formal definition (see also the algebraic conditions in Proposition 2.4 for the corresponding formulation in terms of paths).

\begin{defn}[The group $G^2(\mathbb{R}^d)$]\label{def:GNV}
    The step-2 nilpotent Lie group $G^2(\mathbb{R}^d)\subset\mathbb{R}^d\oplus(\mathbb{R}^{d})^{\otimes 2}$ is defined as follows.
    An element can be presented by a pair $(a,b)\in \mathbb{R}^d\times \mathbb{R}^{d\times d}$
    (that is $a$ is a vector and $b$ is a matrix), the group operation $\otimes$ is defined by
      \begin{equation}\label{eq:group action}
       (a,b)\otimes (a',b') = (a+a',b+b'+a\otimes a'),
      \end{equation}
    and the following condition holds
    \begin{equation}\label{eq:group symmetric condition}
      \forall (a,b)\in G^2(\mathbb{R}^d),\ Sym(b)=\dfrac{1}{2}a\otimes a,
    \end{equation}
    where $Sym(\cdot)$ is the symmetric part of an element, that is $Sym(b)_{i,j}=\frac{1}{2}(b_{i,j}+b_{j,i})$.
    (For clarity, we emphasize that above we used $a\otimes a'= (a_i a'_j)_{i,j}$ for the tensor product).
  \end{defn}
  For an element $(a,b)$, $a$ and $b$ are called the first and the second level, respectively.

  The topology of $G^2(\mathbb{R}^d)$ is induced by the \textit{Carnot-Caratheodory norm} $\norm{\cdot}$,
which gives for an element $(a,b)\in G^2(\mathbb{R}^d)$ the length of the shortest path with bounded variation
that can be ``encoded" as $(a,b)$, i.e.\ whose increment is $a$ and whose iterated integral is $b$.
In other words
\[
\norm{(a,b)}:=\inf\left\{\int_0^1 |\dot{\gamma}(t)|\dif t\, \Big{|}\, \gamma:[0,1]\to\reals^d \text{ is of bounded variation},\, (\gamma_{0,1},S_{0,1}(\gamma))=(a,b)\right\}.
\]
 Showing that the set on which the infimum is taken is non-empty is a non-trivial statement and is the content of Chow's Theorem, see \cite[Theorem 7.28]{friz2010multidimensional}.
The norm defined in this fashion induces a continuous metric $\mathbf{d}$ on $G^2(\mathbb{R}^d)$
through the application
\begin{align}
  \label{def:CarnotCaraDistance}
  \begin{array}{cccc}
  \mathbf{d}: & G^2(\mathbb{R}^d) \times G^2(\mathbb{R}^d) &\to      &\mathbb{R}_{+}\\
     & (g,h)                   &\mapsto    &\norm{g^{-1}\otimes h}
  \end{array}
.\end{align}
$(G^2(\mathbb{R}^d),\mathbf{d})$ is then \emph{a geodesic space}, i.e.\ any two points can
be connected by a geodesic.

  \begin{defn}[Rough paths on $G^2(\mathbb{R}^d)$]\label{def:RoughPaths}
    Let $1/3<\alpha<1/2$. An $\alpha$-H\"older geometric
    rough path on $G^2(\mathbb{R}^d)$ is an element $\mathbf{X}=(X,\mathbb{X})\in\mathcal{C}^{\alpha}([0,T],G^2(\mathbb{R}^d))$.
    More preciesly, $(X_{s,t},\mathbb{X}_{s,t}), s,t\in[0,T],s<t$, are in $G^2(\mathbb{R}^d)$
    and the path is an $\alpha$-H\"older continuous function with respect to the distance $\mathbf{d}$.
  \end{defn}
Without going into details we remark that rough path theory also deals with rough paths which are not geometric, i.e., those for which \eqref{eq:group symmetric condition} does not hold.

An example in the probabilistic setting of an $\alpha$-H\"older geometric
    rough path, for any $\alpha\in(\frac{1}{3},\frac{1}{2})$ is the \emph{Brownian motion rough path},
which is also known as the \emph{enhanced Brownian motion}. It is constructed using Stratonovich integration as follows:
   \begin{align*}
 ({B}_{s,t}, \mathcal{S}_{s,t}) = \Big(B_t - B_s, \int_{s\le u<v\le t} \circ dB_{u}\otimes \circ dB_{v}\Big),\,\, 0\leq s < t.
    \end{align*}

The group structure on $G^2(\mathbb{R}^d)$ and the Carnot-Caradeodory norm and distance are particularly tamed for treating path concatenations.
For example the norm is sub-additive. In particular, for a path $\mathbf{X}=(X,\mathbb{X})$ which takes value in $G^2(\mathbb{R}^d)$
let $\mathbf{X}_{s,t}:= \mathbf{X}_{s}^{-1} \otimes \mathbf{X}_{t}$.
Then for every $s<u<t$
\begin{equation}\label{eq:CC norm subadditive}
    \norm{\mathbf{X}_{s,t}} = \norm{\mathbf{X}_{s,u}\otimes \mathbf{X}_{u,t}} \le  \norm{\mathbf{X}_{s,u}}+ \norm{\mathbf{X}_{u,t}}.
\end{equation}
The next proposition can found be useful for actual estimations. It leans on the equivalence
    \begin{equation}\label{eq:CC norm bound}
    C^{-1} \le \frac{\norm{(a,b)}}{|a|_{\mathbb{R}^d}+|b|_{\mathbb{R}^d\otimes\mathbb{R}^d}^{1/2}}\leq C
    \end{equation}
    for some $C\ge 1$, where $|\cdot|_{\reals^d}$ is the Euclidean norm on $\reals^d$ and $|\cdot|_{\mathbb{R}^d\otimes\mathbb{R}^d}$ is
    the matrix norms induced by the vector norm
  $|\cdot|_{\mathbb{R}^d}$.
%For a pair $(X,\mathbb{X})$ where $X$ is a path in $\reals^d$, we consider $X_{s,t}=:X_t-X_s$ for the increments of $X$
%and $\mathbb{X}_{s,t}$ take value in $\reals^{d\times d}$.

 Assume that $(X_{s,t},\mathbb{X}_{s,t})$, $0\le s<t\le T$ take value in $\reals^d \times \reals^{d\times d}$.
 Define
    \begin{equation}\label{eq:alpha holder norm on pairs}
    |||(X,\mathbb{X})|||_{\alpha}:=||X||_{\alpha}+||\mathbb{X}||_{2\alpha},
    \end{equation}
    where
    \begin{equation}\label{eq:alpha holder norm}
    ||X||_{\alpha}=\sup_{s<t, s,t\in[0,T]}\frac{|X_{s,t}|_{\reals^d}}{|t-s|^{\alpha}} \text{ and } \,\,\,
    ||\mathbb{X}||_{2\alpha}=\sup_{s<t, s,t\in[0,T]} \frac{|\mathbb{X}_{s,t}|_{\reals^d \otimes \reals^d}}{|t-s|^{2\alpha}}.
    \end{equation}

\begin{prop}(\cite[Proposition 2.4]{friz2014course})\label{prop: equiv of rough path norms}
Let $\alpha\in(\frac 13,\frac 12 ]$. $\mathbf{X}=(X,\mathbb{X}):\{0\le s< t\le T\}\to \reals^d \times \reals^{d\times d}$ is a geometric rough path as in Definition \ref{def:RoughPaths} if and only if $X_{s,t}=X_t-X_s$
and the following assumptions hold:
\begin{itemize}
\item $|||(X,\mathbb{X})|||_{\alpha}<\infty$.
\item $\mathbb{X}_{s,t} = \mathbb{X}_{s,u} + \mathbb{X}_{u,t} + X_{s,u}\otimes X_{u,t} $ for every $s<u<t$ (Chen's relation).
 \item $Sym(\mathbb{X}_{s,t})  = \frac{1}{2} X_{s,t}\otimes X_{s,t}$ for every $s<t$ (integration by parts property).
 \end{itemize}
\end{prop}

To end this brief review we mention an alternative definition of the group which has some nice interpretation
in terms of signed area.
A path $\mathbf{X}$ considered in Definition \ref{def:RoughPaths} has
increments in $G^2(\mathbb{R}^d)$. This is relevant for the notion of integration, which is, roughly speaking, defined based on ``sewing" according to the increments.
However, since the symmetric part of the second level depends entirely on the first level by definition, to handle path increments
the following alternative definition for the group on which the rough paths are considered is sometimes more useful.
The corresponding antisymmetric group operation $\wedge$
is defined by
    \begin{align}\label{eq:Area}
 	(a,b)\wedge (a',b') = (a+a',b+b'+ \dfrac{1}{2} (a\otimes a' - a'\otimes a)).
       \end{align}
In particular, unlike the case of the law $\otimes$ where an element $(a,b)$ represents a path
with an increment $a$ and an iterated integral $b$, in the case of the antisymmetric product an element $(a,b)$ represents a path where
$a$ is still an increment but $b$ is now the corresponding area. In other words, for $(X,\mathbb{X})\in G^2(\mathbb{R}^d)$ we consider $(X,A)$ instead, where $A_{s,t}^{i,j}=\frac{1}{2}(\mathbb{X}_{s,t}^{i,j}-\mathbb{X}_{s,t}^{j,i})$.

For example, the Brownian motion considered as a rough path in the case of the antisymmetric product $\wedge$ has the form
\begin{align}
  \label{eq:RoughBrownianAnti}
 \mathbb{B}^{\wedge}_{s,t} = (B_t - B_s, \mathcal{A}_{s,t}),\,\, 0\leq s\leq t,
\end{align}
where $\mathcal{A}$ is the stochastic signed area of $B$, called \textit{the Lévy area}. One also remark that the Lévy area is invariant under performing the integration in either the Stratonovich or the It\^o sense.

%The topology of the antisymmetric group induced as well by the \textit{Carnot-Caratheodory norm} $\norm{\cdot}$,
%which here gives for an element $(a,b)$ the length of the shortest path with bounded variation
%whose increment is $a$ and its area is $b$.

%The reader is also referred to Proposition 2.4 of \cite{friz2014course} stating that the definition of $\alpha$-H\"older rough paths
%can equivalently avoid the algebraic consideration of the group $G^2(\mathbb{R}^d)$ and be postulated directly from paths.

The operation defined in \eqref{eq:Area} has a geometric interpretation which shows why the group is suitable for concatenating paths. The first level translates into path concatenation, whereas the second one gives the law of the ``area concatenation'' (signed area of concatenated paths).
Figure \ref{fig:AreaConcatenation} demonstrates how to calculate the signed area of two concatenated curves.
The areas of $\gamma_1$, $\gamma_2$ and that of the triangle
(formed by the increments of $\gamma_1$ and $\gamma_2$) in the figure correspond respectively
to $b$, $b'$ and $\dfrac{1}{2} (a\otimes a' - a'\otimes a$) in formula \eqref{eq:Area}.
This rule for the area of concatenated paths is also commonly referred as the Chen's rule.
It plays a fundamental role in the theory of
rough paths.

\begin{figure}[!th]
\begin{center}
\includegraphics[scale=0.4]{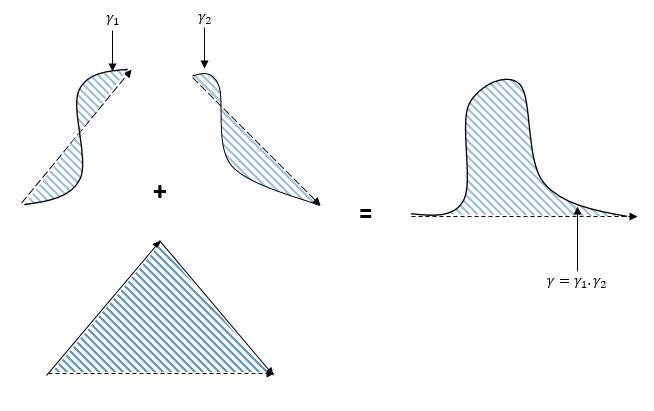}
\caption{A geometric demonstration of Chen's rule on the area of concatenated paths.}
   \label{fig:AreaConcatenation}
\end{center}
\end{figure}

\begin{rem}\label{rem:why do we use G2}
In view of Proposition \ref{prop: equiv of rough path norms}, one can use its assertion as a \emph{definition} for $\alpha$-H\"older rough paths. This is sometimes preferable if one wishes to avoid the Lie group construction. However, in this paper we find the group presentation useful, mainly for the proof of the main result of the paper, Theorem \ref{thm:area anomaly in general conditions}, see section \ref{sec:proof}. Moreover, Step 1 and 4 of the proof are based on \cite{breuillard2009random} which is formulated in the Lie group language. Also, the group actions are useful for presenting computations in a compact way, see in Step 2 of the proof.
\end{rem}

%%%%%%%%%%%%%%%%%%%%%%%%%%%%%%%%%%%%%%%%%%%%%%%%%%%%%%%%%%%%%%%%5
\section{Main result}\label{sec:results}
%%%%%%%%%%%%%%%%%%%%%%%%%%%%%%%%%%%%%%%%%%%%%%%%%%%%%%%%%%%%%%%%%5

For a sequence $X=(X_n)_n$ of elements of $\reals^d$, its continuous rescaled version $X^{(N)}_\cdot$ is defined by
\begin{align*}
X^{(N)}_t=\frac{1}{\sqrt{N}}\big(X_{\intpart{Nt}} + (Nt-\intpart{Nt}) (X_{\intpart{Nt}+1}-X_{\intpart{Nt}})\big).
\end{align*}
We denote the lift of $X^{(N)}$ to a rough path by
\begin{equation}
  \label{eq:Embedding}
\Embedding^{(N)}(X)_{s,t} :=
\Big(X^{(N)}_{s,t}, S_{s,t}(X^{(N)}) \Big),
\end{equation}
where $S_{s,t}(X^{(N)})$ is the second level iterated integral of $X^{(N)}$ between $s$ and $t$, $S_{s,t}:=S^2_{s,t}$, as defined in \eqref{eq:N iterated integral}, and the integration is in the Riemann-Stieltjes sense (which is well-defined since $X^{(N)}$ is of bounded variation on every compact interval $[s,t]$).
%
%Note that as above $\Embedding^{(N)}(X)_t:=\Embedding^{(N)}(X)_{0,t}$.
One can check that
for natural numbers $m<n$, the associated second level iterated integral has the following form
 \begin{equation}\label{eq:discrete iterated integral vs continuous}
S_{m,n}^{i,j}(X^{(1)})= \sum_{m+1\le k < \ell \le n}\Delta X_k^{i} \Delta X_{\ell}^{j} + \frac12 \sum_{m+1\le k \le n}\Delta X_k^{i} \Delta X_{k}^{j},
\end{equation}
where $\Delta X_k:=X_k-X_{k-1}=X_{k-1,k}$ are the increments.

\begin{defn}\label{def:discrete levy area}
For a path $Y$ in $\reals^d$ of bounded variation we define the area
\begin{equation}
A_{s,t}^{i,j}(Y)= \dfrac{1}{2}( S_{s,t}^{i,j}(Y) - S_{s,t}^{j,i}(Y))
\end{equation}
as the antisymmetric part of the iterated integral of $Y$.
Set also $S_t(Y):=S_{0,t}(Y)$ and $A_t(Y):=A_{0,t}(Y)$.
\end{defn}

\begin{defn}\label{def:reg structure}
Let $(X,\F,\pr)$ be a discrete time stochastic process on $\reals^d$.
 We say that $X$ admits a \emph{regeneration structure} if
 there are $ \F$-measurable integer valued random variables $(\tau_k)_{k\in\N_0}$ so that
 $0=\tau_0<\tau_1<\tau_2<...< \infty$ $\pr$-a.s.\
and
\[
\Big(\tau_{k}-\tau_{k-1},\{X_{\tau_{k-1}, \tau_{k-1} + m}: 0 \le m \le \tau_k - \tau_{k-1}\}\Big)
\]
 are independent random variables for $k\ge 1$, and have the same distribution for all $k\ge 2$.
 \end{defn}

\begin{thm}\label{thm:area anomaly in general conditions}
Let $X$ be a discrete time stochastic process on $\reals^d$ with bounded jumps $|X_{n+1}-X_n|_{\reals^d}\le K$ $\pr$-a.s. Assume that $X$ admits a
regeneration structure in the sense of Definition \ref{def:reg structure} and
let $(\tau_k)_{k\ge 0}$ be the corresponding regeneration times.
Assume further that
$X$ satisfies a strong law of large numbers
$$\pr\left(\lim_{n\to\infty}\frac{X_n}{n}= v\right)=1.$$
In this case the speed $v\in \reals^{d}$ is defined by
$$v:=\frac{\E[ X_{\tau_1,\tau_2}]}{\E[\tau_2-\tau_1]}.$$
Also, assume that
$\bar X_n =X_n-nv$ satisfies an annealed invariance principle with covariance matrix
\[
M= \frac{\E[ \bar X^{(1)}_{\tau_1,\tau_2} \otimes \bar X^{(1)}_{\tau_1,\tau_2}]}{\E[\tau_2-\tau_1]}.
\]
Last assumption is the following moment condition:
\begin{equation}
   \label{eq:momentcondition}
\E[(\tau_k-\tau_{k-1})^{2p}]<\infty
\end{equation}
for some $p\ge 4$.
Then we have the following weak convergence with respect to $\pr$
to
\[
\Embedding^{(N)} (\bar X) \Rightarrow (B_{s,t},\mathcal{S}_{s,t} + (t-s) \Gamma)_{0\le s<t\le T} \,\text{  in   }\, \mathcal{C}^{\alpha}\left([0,T],G^2(\mathbb{R}^d)\right)
\]
 for all $\alpha\in(\frac{1}{3},\frac{p^*-1}{2p^*})$,
 where $p^*=\min\{\lfloor p \rfloor,2 \lfloor p/2 \rfloor \}$,
 and the couple $(B,\mathcal{S})$ are the Brownian motion with covariance matrix $M$ and its
 second level iterated Stratonovich integral process.
Moreover, the correction is the antisymmetric matrix
$$
\Gamma=\frac{\E[A_{\tau_1,\tau_2}(\bar X^{(1)})]}{\E[\tau_2-\tau_1]}.
$$
In particular, if the moment condition
holds true for all $1\le p<\infty$ then the convergence holds true for all $\alpha<\frac{1}{2}$.
\end{thm}

\begin{rem}
The corresponding result for the area with the same correction $\Gamma$ holds as well, that is
whenever the path is considered with the antisymmetric operation, $S(X)$ is replaced by $A(X)$ and
the enhanced Brownian motion $\mathcal{S}$ is replaced by the Stratonovich Levy area $\mathcal{A}$.
\end{rem}
The correction matrix $\Gamma$ has the following decomposition.
\begin{lem}
The following decomposition holds
\[
\Gamma\; =\; \frac{\E[A_{\tau_1,\tau_2}(X^{(1)})]}{\E[\tau_2-\tau_1]} \; + \;
\dfrac{\E\left[\sum_{\tau_1 < k<\ell \leq \tau_2}(v\otimes \Delta X_k + \Delta  X_{\ell}\otimes v) - (\Delta X_k \otimes v + v \otimes \Delta  X_{\ell})\right]}{2\E[\tau_2-\tau_1]}.
\]
\end{lem}

\begin{proof}
One has
 \begin{align*}
   \Delta \bar X_{\ell}\otimes \Delta \bar X_k =
       \Delta X_{\ell}\otimes \Delta X_k + v^{\otimes 2} - (v\otimes \Delta X_k + \Delta  X_{\ell}\otimes v)
 \end{align*}
since $\bar X_n = X_n - nv$. Neglecting the symmetric term we get the assertion %decompose $\Gamma$ as
%  %  \begin{equation}  \label{eq:correction}
%       $\Gamma = \Gamma +  R$
%  %  \end{equation}
% where
% \[
% R = \dfrac{\E\left[\sum_{\tau_1 < k<\ell \leq \tau_2}(v\otimes \Delta X_k + \Delta  X_{\ell}\otimes v) -
%       (\Delta X_k \otimes v + v \otimes \Delta  X_{\ell})\right]}{2\E[\tau_2-\tau_1]}
% \]
%and
% \[
% \Gamma = \dfrac{\E\left[\sum_{\tau_1 < k<\ell \leq \tau_2}\Delta  X_{\ell}\otimes \Delta X_k - \Delta X_k\otimes \Delta X_{\ell}\right]}{2\E[\tau_2-\tau_1]}
% \]
% is the expected area of $X^{(1)}$ in the time interval $[\tau_1,\tau_2]$ divided by the interval expected size.
%
\end{proof}

%%%%%%%%%%%%%%%%%%%%%%%%%%%%%%%%%%%%%%%%%%%%%%%%%%%%%%%%%%%%%%%%%%%55
\section{Simple applications}\label{sec:examples}
%%%%%%%%%%%%%%%%%%%%%%%%%%%%%%%%%%%%%%%%%%%%%%%%%%%%%%%%%%%%%%%%%%%55
In this section we construct processes lacking or having non-zero area anomaly with a simple but instructive description. For starters, going back to the two processes compared in Chapter 2, the process with four steps clockwise loops every two steps (see Figure \ref{fig:SRWLoop}) have a non-zero area anomaly, while the process which stands still for four steps every two steps have no correction.
For a discrete time process $X$, we remind the reader the notation $X^n(t):=\frac{ X_{\lfloor nt\rfloor}}{\sqrt{n}}+ \frac{1}{\sqrt{n}}\big((nt-\lfloor nt\rfloor)( X_{\lfloor nt\rfloor+1}- X_{\lfloor nt\rfloor})\big)$ .

%%%%%%%%%%%%%%%%%%%%%%%%%%%%%%%%%%%%%%%%%%%%%%%%%%%%%%%%%%%%%%%%%%%55
\subsection{Rough path version of Donsker's Theorem \cite{breuillard2009random,kelly2016rough}}
%%%%%%%%%%%%%%%%%%%%%%%%%%%%%%%%%%%%%%%%%%%%%%%%%%%%%%%%%%%%%%%%%%%55
Consider a discrete time random walk $X$ on $\reals^d$. Assume that
the increments $\Delta X_{n}$ are i.i.d.\ non-zero centered with finite $2p$ moment for some $p\ge4 $.
%$\zeta_{1}= \begin{cases}
%                     e_1 &  w.p.\,p/2 \\
%                     -e_1 &  w.p.\,(1-p)/2 \\
%                     \pm e_2 &  w.p.\,1/4
%                    \end{cases}$.
Then
$\Embedding^{(N)}(X)_\to (B,\mathcal{S})$ in distribution in $\mathcal{C}^{\alpha}\left([0,1],G^2(\mathbb{R}^2)\right),\alpha\in(\frac{1}{3},\frac{1}{2}-\frac{1}{p^*})$, where $p^*=\min(\lfloor p \rfloor , 2 \lfloor p/2 \rfloor )$, $B$ is a $d$-dimensional Brownian motion, and $(\mathcal{S})$ is the Stratonovich second level iterated integral of $B$ in the time interval $[0,1]$. One can see this as a special case of our theorem for the case the regeneration times are the set of natural numbers.

Note that in this example no area correction appears. Since centering a random walk with a drift defines a new random walk with no drift, this example shows that a non-zero area anomaly cannot be created only from the presence of drift.
\subsection{Rotating drift \cite{lopusanschi2017levy}}
%%%%%%%%%%%%%%%%%%%%%%%%%%%%%%%%%%%%%%%%%%%%%%%%%%%%%%%%%%%%%%%%%%%55
The following example shows that non-zero area anomaly is possible even with no speed.
Consider $\Z^2\subset\mathbb{C}$.
Let $(\zeta_n)_n$ be i.i.d.\ so that $P(\zeta_1=1)=p=1-P(\zeta_1=-1)$.
Define $\Delta X_n:=i^n \zeta_n$, $i=\sqrt{-1}$.
Then $X^N \to B$ in distribution in the uniform topology, $B$ is a BM with covariance $2p(1-p)\mathcal{I}$, where
$\mathcal{I}$ is the identity matrix.
However, after rescaling $\Embedding^{(N)}(X) \to (B,\mathcal{S} + \Gamma \cdot)$
in distribution in $\mathcal{C}^{\alpha}\left([0,1],G^2(\mathbb{R}^2)\right),\alpha<1/2$,
where
$\mathcal{S}$ is the Stratonovich second level iterated integral of $B$ in $[0,1]$.

Indeed, $X$ has a regeneration structure for the deterministic times $\tau_k:=4k,k\ge0$, which, trivially, have all moments. One can check that there is a strong law of large numbers with speed $v=0$. Then $\Gamma=\frac{1}{4}\E [A_{0,4}(X)]$, and straight forward computation yields
\[
                                \Gamma= \frac {(2p-1)^2}{4} \left(
                                                \begin{matrix}
                                                0  & 1 \\
                                                -1 & 0 \\
                                                \end{matrix}
                                            \right)
                                        \]
which is non-zero if $p\ne \frac 1 2$.
%%%%%%%%%%%%%%%%%%%%%%%%%%%%%%%%%%%%%%%%%%%%% new frame %%%%%%%%%%%%%%%%%%

%%%%%%%%%%%%%%%%%%%%%%%%%%%%%%%%%%%%%%%%%%%%% new frame %%%%%%%%%%%%%%%%%%
\subsubsection*{Different presentation: non-elliptic periodic environment}
%%%%%%%%%%%%%%%%%%%%%%%%%%%%%%%%%%%%%%%%%%%%%%%%%%%%%%%%%%%%%%%%%%%55
The same example as above can be presented as a walk in a certain fixed space non-homogenous environment rather than a walk with drift rotating in time. Consider again $\Z^2$ as a subset of $\mathbb{C}$ and fix some $0<p<1$.
 Let $\omega$ be the two-periodic environment given by:
$\om(0,1)=p =1- \om(0,-1)$,
$\om(1,1+i)=p =1- \om(1,1-i)$,
$\om(1+i,i)=p =1- \om(1+i,2+i)$,
and $\om(i,0)=p =1- \om(i,2i)$.
% \begin{itemize}
%\item[] $\om(0,1)=p =1- \om(0,-1)$,
%\item[] $\om(1,1+i)=p =1- \om(1,1-i)$,
%\item[] $\om(1+i,i)=p =1- \om(1+i,2+i)$,
%\item[] $\om(i,0)=p =1- \om(i,2i)$.
% \end{itemize}
In particular, two-periodicity means that $\om(v,w)= \om(v+z,w)= \om(v,w+z) $ for every $w,v\in \Z+\Z i$ and $z\in 2\Z+ 2\Z i$.
Finally, let $X$ be the Markov chain on $\Z^2$ with transition probabilities $\omega$. Then, by parity, the law of $X$ is the same as the law of the last rotating drift example. In particular, the same result holds for this example as well.
%%%%%%%%%%%%%%%%%%%%%%%%%%%%%%%%%%%%%%%%%%%%% new frame %%%%%%%%%%%%%%%%%%

\section{Applications}\label{applications}
\subsection{Random walks in random environment}
We first define random walks in random environment on $\Z^d$.
Let $\mathcal{E}:=\{e_i:i=1,...,2d\}\subset\Z^d$ be the set of neighbors of the origin.
Let $\mathcal{P}_x$ be the space of probability distributions on the algebraic sum $x+\mathcal{E}:=\{x+e:e\in\mathcal{E}\}$. We call
$\Om=\prod_{x\in\Z^d}\mathcal{P}_x$ the space of environments on $\Z^d$. In particular, an environment $\om\in\Om$ is of the form
$\om=(\om(x,x+e))_{x\in\Z^d,e\in\mathcal{E}}$ so that $\om(\cdot,\cdot)\ge 0$ and $\sum_{e\in\mathcal{E}}\om(\cdot,\cdot + e)=1$.

For a fixed environment $\om\in\Om$ and a starting point $x\in\Z^d$ we define a nearest neighbor walk $X$
on $\Z^d$ to be the Markov chain starting at $x$, $P_{x,\om}(X_0=x)=1$, with transition probabilities $P_{x,\om}(X_{n+1}=y+e|X_n=y)=\om(y,y+e)$.
Given a probability distribution $P$ on $\Om$, the \emph{annealed} (and sometimes called also the \emph{averaged}) law of the walk $X$ is characterized by
$\pr_{x}(\cdot):=\int P_{x,\om}(\cdot) \dif P(\om)$.
We also call $P_{x,\om}$ the \emph{quenched} law.
We say that the environment is i.i.d.\ if $(\om_x)_{x\in\Z^d}$ is an i.i.d.\ sequence under $P$.
An i.i.d.\ random environment is called \emph{uniformly elliptic} if there is some
deterministic $\kappa>0$ so that $P(\omega(0,e)\ge\kappa$ for all $e\in\mathcal{E})=1$.

We now define some ballisticity conditions and for that adapt the notation of \cite{berger2014effective}.
Fix $L \ge 0$ and let $\ell \in \mathbb S^{d-1}$ be an element of the unit sphere. Then we write
\begin{equation} \label{eq:exitTimeOneDefX}
 H_L^\ell := \inf\{ n \in \N_0 : X_n \cdot \ell > L\}
\end{equation}
for the first entrance time of $(X_n)$ into the half-space $ \{x \in \Z^d : x \cdot \ell > L\},$
where ${\N_0 = \{0, 1, 2, \ldots\}.}$

\begin{defn}[Sznitman $(T'){|\ell}$ condition \cite{sznitman2002effective}] \label{def:Tgamma}
Let $\gamma\in (0,1]$ and $l\in\mathbb S^{d-1}$.
We say that \emph{condition $(T)_\gamma$} is satisfied with respect to $\ell$, and write $(T)_\gamma|\ell$, if for every $b>0$ and each $\ell'$ in some neighborhood of $\ell$ one has that
\begin{align*}
\limsup_{L \to \infty} L^{-\gamma} \ln \pr_0 \big( H_L^{\ell'} > H_{bL}^{-\ell'} \big) < 0.
\end{align*}
We say that \emph{condition $(T')$} is satisfied with respect to $\ell$, and write
$(T')|\ell$, if condition $(T)_\gamma|\ell$
is fulfilled for every $\gamma\in (0,1)$.
\end{defn}

\begin{defn}[Berger-Drewitz-Ramirez $(P^*_M \vert \ell$) condition \cite{berger2014effective}]\label{def:PCond}
Let $M > 0$ and $\ell \in \mathbb S^{d-1}.$
 We say that {\em condition $P^*_M \vert \ell$} is satisfied
with respect to $\ell$
if
for every $b > 0$ and all $\ell' \in \mathbb S^{d-1}$ in some neighborhood of $\ell,$ one has that
\begin{equation} \label{eq:PCond}	
\limsup_{L \to \infty} L^{M} P_0 \big( H_{bL}^{-\ell'} < H_{L}^{\ell'} \big) = 0,
\end{equation}
\end{defn}

\begin{thm}\label{thm:area anomaly anneald rwre}
Let $X$ be a random walk in i.i.d.\ and uniformly elliptic random environment on $\mathbb{Z}^d$, where $d\ge 2$.
Let $\ell\in\mathbb{S}^{d-1}$ and assume that the Sznitman-type condition $P^*_M \vert \ell$ of
Berger-Drewitz-Ramirez holds for some $M>15d+5$.
%(Assume furthermore Sznitman condition $(T)$ in some direction $\ell$).
Then the conditions of Theorem \ref{thm:area anomaly in general conditions} are satisfied, and moreover
the moment condition holds for all $1\le p<\infty$.
\end{thm}
\begin{proof}
Berger, Drewitz, and Ramirez \cite[Theorem 1.6]{berger2014effective} states that in this case the stronger condition $(T'){|\ell}$ of Sznitman also holds.
The law of large numbers, including the existence of regeneration times were proved in \cite{sznitman1999law}
where the independence mentioned only the increments $(X_{\tau_{k-1},\tau_k},\tau_k-\tau_{k-1})$.
However, the proof of \cite{sznitman1999law} shows that the walk on different intervals
$(X_{\tau_{k-1}, \tau_{k-1}+ m}\cdot\ell)_{m\le \tau_k-\tau_{k-1}}$, is independent for $k\ge1$ and identically distributed for
$k\ge 2$, and, moreover, the walk satisfies $X_{\tau_{k-1}, \tau_{k-1}+ m}\cdot\ell>0$ for all $m>0$. This form appears specifically, e.g., in \cite[Claim 3.4]{berger2008limiting}. In particular, $X$ admits a regeneration structure
in the sense of Definition \ref{def:reg structure}.
Annealed invariance principle was proved in \cite[Theorem 4.1]{sznitman2000slowdown} and \cite[Theorem 3.6]{sznitman2001class}
 based on the finiteness of all moments for the regeneration time, which was proved in \cite[Theorem 3.4]{sznitman2001class}.
\end{proof}

\begin{rem}\label{rem:all moments conjecture}
 A version of a well-known conjecture by Sznitman is as follows: For random walks in random environment on $\mathbb{Z}^d$, $d\ge 2$, in i.i.d.\ and uniformly elliptic environment directional transience in some direction $\ell$ is enough for attaining finiteness of \emph{all} moments for the regeneration times.
  Therefore, assuming the conjecture then directional transience in some direction $\ell$ is enough for an
  annealed convergence in the $\alpha$-H\"older rough path topology for all $\alpha<1/2$. In particular, one would not expect an example with a more singular convergence, or, more accurately no example for directionally transient i.i.d uniformly elliptic RWRE for which there are some $\alpha<\beta< 1/2$ so that the convergence holds in $\alpha$-H\"older but not in $\beta$-H\"older.
  \end{rem}

A Dirichlet distribution with parameters $\alpha_1>0,...,\alpha_N>0$ is defined by the density $\varphi$ with respect to Lebesgue measure
on $S^{N-1}:=\{x\in\reals^N:x_i\ge 0, \sum_{i=1}^{N} x_i =1\}$, the $(N-1)$-dimensional simplex, defined by
\[
\varphi(x)= \frac{1}{B(\alpha_1,...,\alpha_N)}\prod_{i=1}^{N}x^{\alpha_1-1},
\]
where $B(\alpha_1,...,\alpha_N)$ is a normalizing constant.
Let $X$ be a random walk in i.i.d.\ random environment so that $\omega_0$ has the Dirichlet distribution with parameters $\alpha_e>0$ for $e\in\mathcal{E}$. Let
\[
\kappa=2\sum_{e\in\mathcal{E}}\alpha_e - \max_{e\in\mathcal{E}}\{\alpha_e+\alpha_{-e}\}.
\]
It is known that in dimension $d\ge 2$ if
\begin{equation}\label{eq:condition for ballisticity for Dirichlet}
\sum_{e\in\mathcal{E}}|\alpha_e-\alpha_{-e}|>1
\end{equation}
then for every direction $\ell$ for which $\sum_{|e|=1}\alpha_e e\cdot \ell >0$
 there is a decomposition of the walk to regeneration intervals in direction $\ell$ in the form that appears in the proof of Theorem \ref{thm:area anomaly anneald rwre} above, and in particular it admits a regeneration structure in the sense of Definition \ref{def:reg structure}. Moreover, the regeneration interval $\tau_2 - \tau_1$ has a finite $p$-th moment if and only if $p<\kappa$, see \cite[Corollary 2]{sabot2016random}. In particular, we have
    \begin{thm}\label{thm:area anomaly anneald Dirichlet rwre}
        Let $X$ be a random walk in i.i.d.\ Dirichlet environment $\mathbb{Z}^d$, where $d\ge 2$. Assume that \eqref{eq:condition for ballisticity for Dirichlet} holds with the parameter $\kappa$ defined above. Assume in addition that $\kappa>8$. Then the conditions of Theorem \ref{thm:area anomaly in general conditions} are satisfied with $p<\kappa/2$.
        In particular, we have a convergence in the $\alpha$-H\"older rough path topology for all $\alpha\in(\frac{1}{3},\frac12 -\frac{1}{(\kappa/2)^*})$, where $(\kappa/2)^*=\min\{\lfloor \kappa/2 \rfloor,2 \lfloor \kappa/4 \rfloor \}$.
    \end{thm}

%    We have just presented an example of a class of RWRE which admit a regeneration
%    structure so that the regeneration times have no infinitely many moments, but the parameters could be tuned to have any prescribed moment, hence our theorem would give a convergence in the $\alpha$-H\"older rough path topology but for $\alpha<1/2$ which are \emph{bounded away} from $1/2$.

    \begin{rem}
    It is relevant to point out here that the $\alpha$-H\"older rough topology is not the only choice one can make (although it is certainly more common). We chose to work with it in this paper due to availability of the results of \cite{breuillard2009random} and \cite{kelly2016rough} which were considered in these settings. However, without going into the details here, let us mention that one can also define a rough path topology using the $p$-variation norm, which is parameterization-free and corresponds to $1/p$-H\"older topology. This was in fact the original definition in \cite{lyons1998differential}. Using some recent available estimates, we believe that one should be able to prove a version of our Theorem \ref{thm:area anomaly in general conditions} in the $p$-variation rough path topology, for every $p>2$, assuming only finiteness of the second moment of the jumps. The last example shows why this might be desirable. On the other hand, in the view of Remark \ref{rem:all moments conjecture}, there's no advantage for $p$-variation rough paths if one is interested in RWRE from the ballistic class.
    \end{rem}

  We close this section with an open problem. As one can notice in the examples given in Chapter \ref{sec:examples},
  to construct a law with non-zero area anomaly it is not enough to have an asymptotic direction or non-trivial covariances.
  Area anomaly might hint that there is some asymmetry in the shape of the path with respect to the asymptotic direction.
  We conjecture that, roughly speaking, any ``reasonably asymmetric" RWRE from the ballistic class considered in Theorem \ref{thm:area anomaly anneald rwre} would have a non-zero area anomaly. However, the following is still an open problem.
  \begin{prob}
   Is there a RWRE satisfying the conditions of Theorem \ref{thm:area anomaly anneald rwre} for which
   the area anomaly $\Gamma$ is non-zero? Note that the question is open even for stationary and ergodic RWRE.
  \end{prob}

\subsection{Periodic graphs or hidden Markov walks}
Theorem \ref{thm:area anomaly in general conditions} naturally generalizes the main results in \cite{lopusanschi2017levy} and \cite{lopusanschi2017area}.

\begin{thm}[\cite{lopusanschi2017area} and \cite{lopusanschi2017levy}]\label{thm:area anomaly periodic graphs or hidden markov walks}
Let $X$ be either an irreducible Markov chain on a periodic graph (see the definition in \cite{lopusanschi2017levy})
or an irreducible hidden Markov walk driven by a finite state Markov chain (see the definition in \cite{lopusanschi2017area}), then
the conditions of Theorem \ref{thm:area anomaly in general conditions} are satisfied.
\end{thm}
\begin{proof}[Proof (Sketch)]
If $(Y_n)_n$ is an irreducible Markov chain on a periodic graph or an irreducible hidden Markov walk,
it admit an underlying irreducible Markov chain $(X_n)_n$ on a finite state space.
More precisely, for every $n\geq 1$, the increment $Y_{n+1}-Y_n$ depends on $X_n$ in an appropriate way.

We can thus define a sequence on stopping times for $(X_n)_n$ as
\begin{align*}
 & T_0 = 0 \text{ and }
 T_n = \inf\{k>T_{n-1}: X_k = X_0\},  n\geq 1.
\end{align*}
In particular, it is a sequence of return times to the initial position of $(X_n)_{n\ge 0}$.
By construction, the sequence $(T_n)_{n\ge 0}$ is strictly increasing and, as $(X_n)_{n\ge 0}$ is irreducible,
all $T_n$ are finite a.s.\ The increments $(T_{n+1}-T_n)_{n\ge 0}$ are i.i.d., as well as the variables
$(Y_{T_{n+1}}-Y_{T_n})_{n\ge 0}$ (see the proof in \cite{lopusanschi2017levy}) and, more generally,
\[
\left((Y_{T_n + m}-Y_{T_n})_{0\le m\le T_{n+1}-T_n}, T_{n+1}-T_n \right)_{n\ge 0}.
\]
Consequently the process $(Y_n)_{n\ge 0}$ admits a regeneration structure.

Moreover, since $(X_n)_{n\ge 0}$ is irreducible and takes values on a finite state space, all moments of the
increments $T_{n+1}-T_n$ are finite (they actually have geometric tails).
Concluding the law of large numbers and the invariance principle is now routine.
 \end{proof}

%%%%%%%%%%%%%%%%%%%%%%%%%%%%%%%%%%%%%%%%%%%%%%%%%%%%%%%%%%%%%%%%5
\section{Proof of Theorem \ref{thm:area anomaly in general conditions}}\label{sec:proof}
%%%%%%%%%%%%%%%%%%%%%%%%%%%%%%%%%%%%%%%%%%%%%%%%%%%%%%%%%%%%%%%%%5

We shall take the general route of \cite{lopusanschi2017area}, where the authors proved first the convergence
for the path on a sequence of return times with exponential tails, and then moved to the full
path, where they identified an area correction. For both identification of the limit and tightness they used the strong Markov property together with the the tail bounds of the stopping times. To demonstrate the idea in a rather simple way the reader is suggested to think about the case of random walks on a deterministic periodic environment on $\Z^d$, where the decomposition is done according to return times of the walk to the origin \emph{modulu the period}.
In our proof, we decompose the path according to the regeneration times, which are not stopping times and therefore the strong Markov property does not apply. However, as we shall show, the i.i.d.\ nature of our decomposition together with the finiteness of the regeneration time interval moments
are enough to conclude.

\begin{proof}[Proof of Theorem \ref{thm:area anomaly in general conditions}]

%The main difficulty would be that the $\tau_k$ are not stopping times and hence strong Markov property is not %applicable.
%However $(\tau_k-\tau_{k-1})_k$ are i.i.d\ and all of their moments are finite, which is enough for the argument.

The proof will be divided in four steps:
\begin{itemize}
\item  Convergence in distribution of the centered discrete process
given by the sum of $\bar X_{\tau_k,\tau_{k+1}}$ using the rough path version of Donsker's Theorem.
%from \cite{breuillard2009random}. In fact, we use the generalization of \cite{kelly2016rough} to jumps who has a general covariance matrix.
\item Convergence of the finite-dimensional marginals of the subsequence $\left(\Embedding^{(\tau_k)}(\bar X)\right)_{k\geq 1}$,
where we see the area anomaly $\Gamma$.
\item Convergence of finite-dimensional marginals of the full process $\left(\Embedding^{(N)}(\bar X)\right)_{N\ge 1 }$.
\item Tightness of the sequence $\left(\Embedding^{(N)}(\bar X)\right)_{N\geq 1}$.
\end{itemize}

%%%%first step
\textbf{Step 1:}
Let $Y_n=:\bar X_{\tau_{n}}$.
We claim that $\Embedding^{(N)}(Y)_{t\leq T} \to (B'_t,\mathcal{S}'_t)_{t\leq T}$ in distribution with respect to $\pr_0$ in $\mathcal{C}^{\alpha}([0,T],G^2(\mathbb{R}^d))$ for all $\alpha\in(\frac{1}{3},\frac{1}{2}-\frac{1}{2p^*})$,
where $B'$ is a Brownian motion with covariance matrix ${\E[ \bar X_{\tau_1,\tau_2} \otimes \bar X_{\tau_1,\tau_2}]}$ and
$\mathcal{S}'$ is its corresponding second level iterated Stratonovich integral.
Indeed, assuming without loss of generality that $(\bar X_{\tau_{1}},\tau_1)$ has the same distribution of $(\bar X_{\tau_{1},\tau_2},\tau_2-\tau_1)$, then
$Y_n=\sum_{i=1}^n \Delta Y_i$ is a sum of i.i.d.\ centered random variables
with values in $\reals^d$ and with covariance $\E[ \Delta Y_1 \otimes \Delta Y_1]={\E[ \bar X_{\tau_1,\tau_2} \otimes \bar X_{\tau_1,\tau_2}]}$.
Moreover, since the jumps are $\pr_0$-a.s.\ bounded $|\Delta X_n|_{\reals^d}\le K$, then $|\Delta Y_n|_{\reals^d}\le R (\tau_n-\tau_{n-1})$ and therefore also have finite $2p$ moment, where $R=R(K,d)$ is some constant.
Applying Theorem 1 of \cite{breuillard2009random} to the process $D^{-1/2}Y_n$, where $D={\E[ \bar X_{\tau_1,\tau_2}\otimes \bar X_{\tau_1,\tau_2}]}$, we get weak convergence of $Y^{(N)}$ in in $\mathcal{C}^{\alpha}\left([0,T],G^2(\mathbb{R}^d)\right)$ for all $\alpha\in(\frac{1}{3},\frac{1}{2}-\frac{1}{2p^*})$.
(Alternatively, Lemma 3.1 of \cite{kelly2016rough} with $V=1$ in the equation appearing there implies the convergence in uniform topology and therefore the convergence of the finite-dimensional marginals, then the tightness in $\mathcal{C}^{\alpha}\left([0,T],G^2(\mathbb{R}^d)\right)$ for all $\alpha\in(\frac{1}{3},\frac{1}{2}-\frac{1}{2p^*})$
is showed in the proof of that lemma using the Kolmogorov Criterion.)

%%%%% second step

\textbf{Step 2:}
Denote by $\delta_\epsilon$ the standard dilatation by $\epsilon$, that is $\delta_\epsilon (x,a)=(\epsilon x, \epsilon^2 a)$.
By \eqref{eq:group action} and \eqref{eq:discrete iterated integral vs continuous} we have the following decomposition of the rough path lift of $\bar X$
    \begin{align*}
        \delta_{N^{1/2}}\Embedding^{(N)}(\bar X)_{\frac{m}{N}} =
        \bigotimes_{k=1}^{m} \left(\Delta \bar X_k,\dfrac{1}{2}\Delta\bar X_k^{\otimes2}\right)
    \end{align*}
%\otimes
%\delta_{Nt-\intpart{Nt}}\left(\Delta \bar X_{\intpart{Nt}+1},\dfrac{1}{2}\Delta\bar X_{\intpart{Nt}+1}^{\otimes 2}\right)

Then, using the properties of integrals for piecewise linear processes, for $r\in\mathbb{N}$, we get the decomposition
     \begin{align} \label{eq:Decomposition}
        \delta_{\tau_r^{1/2}}\Embedding^{(\tau_r)}(\bar X)_1=
        \bigotimes_{k=1}^{r} \left(\Delta Y_k,\dfrac{1}{2}\Delta Y_k^{\otimes 2}\right)
        \otimes  \bigotimes_{k=1}^{r} (0,a_k),
    \end{align}
 where
%$\Delta (\bar X_{\tau})_k = \bar X_{\tau_k} - \bar X_{\tau_{k-1}}$ and
    \begin{align*}
  a_k = \dfrac{1}{2} \sum_{\tau_{k-1}+1\leq m<n\leq \tau_k} \left(\Delta \bar X_n\otimes \Delta\bar X_m - \Delta\bar X_m\otimes \Delta\bar X_n\right)
    \end{align*}
  is the discrete area between the times $\tau_{k-1}$ and $\tau_k$.
  We note that the first term in the product at the right hand side of \eqref{eq:Decomposition}
  corresponds to the rough path of a partial sum of our i.i.d.\ variables  $\Delta(\bar X_{\tau})_{k}$. We have seen in step $1$ that the sequence of rough paths which is corresponding to these partial sums converges in distribution to the enhanced Brownian motion in the $\alpha$-H\"older topology, which implies that the corresponding finite-dimensional marginals converge in distribution to those of the Brownian motion.

On the other hand, for every fixed $s\in\N$ and $0<t_1<\ldots<t_s$, using the fact that the process $X$ admits a regeneration structure, we conclude that $a_k,k\ge 2,$ are i.i.d., and moreover each coordinate of $a_k$ is bounded by a multiple of $K(\tau_k-\tau_{k-1})^2$, which has a bounded expectation. Thus, by the law of large numbers, we have the following convergence
  \begin{align*}
      \left(\dfrac{1}{r}\sum_{k=1}^r a_{\lfloor t_1k \rfloor} ,...,\dfrac{1}{r}\sum_{k=1}^r a_{\lfloor t_sk \rfloor}\right) \underset{r\to \infty}{\longrightarrow} \E[a_2](t_1,...,t_s) \text{ a.s.}
  \end{align*}
Moreover, the law of large numbers implies $\frac{\tau_k}{k} \to \E[\tau_2-\tau_1]=:\beta$ $\pr_0$-a.s.
 Since $\bigotimes_{k=1}^{r} (0,a_k) = (0,\sum_{k=1}^r a_k)$,
 we can use Slutsky's theorem \cite{slutsky1923stochastische} as in
 \cite[Lemma 2.3.2]{lopusanschi2017area} to conclude that we have the following convergence in distribution
   \begin{align*}
  \left( \Embedding^{(\tau_r)}( \bar X)_{t_1},...,\Embedding^{(\tau_r)}( \bar X)_{t_s}
 \right) \underset{r\to \infty}{\longrightarrow} \left((B_{t_1},\mathcal{S}_{t_1} + t_1\Gamma),\ldots,(B_{t_s},\mathcal{S}_{t_s} + t_s\Gamma)\right)
   \end{align*}
 where $\Gamma=\beta^{-1}\E[a_2]$ is an antisymmetric matrix, $B=\beta^{-1/2} B'$ and $\mathcal{S}$ is
 its corresponding Stratonovich iterated integral.

 %Note that in order to have $S=0$ we do not require $\E[\Delta X_n] = 0$ for all $n$, we only need $\E[X_{\tau_2}-X_{\tau_1}]=0$ i.e. centered increments of the renewal times.

%%%%% third step

 \textbf{Step 3:}
 Set
 %$\mathbf{\bar X}_n=\left(\bar X_{n}, S_{n}(\bar X)\right)$ and
 $\kappa(n)$ to be the unique integer
 such that $\tau_{\kappa(n)}\leq n<\tau_{\kappa(n)+1}$.
 We use \eqref{def:CarnotCaraDistance} together with the fact that $\bar X$ has bounded increments a.s.\ to deduce
 \begin{eqnarray*}
\mathbf{d}\left(\delta_{N^{-1/2}}\left(\Embedding^{(1)}(X)_{\tau_{\kappa(\lfloor Nt \rfloor)}}\right),\delta_{N^{-1/2}}\left(\Embedding^{(1)}(X)_{\intpart{Nt}}\right)\right)
% \norm{\delta_{N^{-1/2}}\left(\Embedding^{(1)}(X)_{\tau_{\kappa(\intpart{Nt})}}^{-1} \right) \otimes \delta_{N^{-1/2}} \left(\Embedding^{(1)}(X)_{\intpart{Nt}}\right)}
  &=& N^{-1/2}\mathbf{d}\left(\left(\Embedding^{(1)}(X)_{\tau_{\kappa(\lfloor Nt \rfloor)}}\right),\left(\Embedding^{(1)}(X)_{\intpart{Nt}}\right)\right) \\
  &\le&   d K N^{-1/2} (\intpart{Nt} - \tau_{\kappa(\intpart{Nt})}).
 \end{eqnarray*}
%, which, together with the bound \eqref{eq:CC norm bound} on $G^2(\mathbb{R}^d)$,
%to deduce
%  \begin{equation}
%     \label{eq:bound}
%     \norm{\left(\Embedding^{(1)}(X)_{\tau_{\kappa(\intpart{Nt}}}^{-1} \right)\otimes  \left(\Embedding^{(1)}(X)_{\intpart{Nt}} \right)}
%          \leq   d K^2  (\intpart{Nt} - \tau_{\kappa(\intpart{Nt})})^2.
%  \end{equation}
Applying the Markov inequality we obtain the following convergence for any $\epsilon>0$
    \begin{align*}
 & \mathbb{P}\left(\mathbf{d}\left(\delta_{N^{-1/2}}\left(\Embedding^{(1)}(X)_{T_{\kappa(\lfloor Nt \rfloor)}}\right),
 \delta_{N^{-1/2}}\left(\Embedding^{(1)}(X)_{\intpart{Nt}}\right)\right) >\epsilon\right)
      \leq
      Kd\dfrac{\E[\intpart{Nt} - \tau_{\kappa(\intpart{Nt})}]}{N^{1/2}\epsilon}
   %     \leq
   %       \dfrac{\tilde{K} \E[(\tau_2-\tau_1)^2]}{N^{1/2}\epsilon}
%   \underset{N\to\infty}{\longrightarrow} 0,
    \le
Kd\dfrac{\E[\tau_{\kappa(\intpart{Nt})+1} - \tau_{\kappa(\intpart{Nt})}]}{N^{1/2}\epsilon}.
     \end{align*}
Note that $\E[\tau_{\kappa(\intpart{Nt})+1} - \tau_{\kappa(\intpart{Nt})}]\le (\E[\tau_{2}^3]t)^{1/3}  N^{1/3}$ as $\tau_2$ has a finite third moment.
Indeed,
\begin{align*}
\E[\tau_{\kappa(\intpart{Nt})+1} - \tau_{\kappa(\intpart{Nt})}]
& = \sum_{k=0}^{Nt} \E[(\tau_{k+1} - \tau_{k})\one_{\kappa(\intpart{Nt})=k}]\\
& \le \sum_{k=0}^{Nt} \E[(\tau_{k+1} - \tau_{k})^3]^{1/3}\pr({\kappa(\intpart{Nt})=k})^{2/3}\\
& \le \E[\tau_{2}^3]^{1/3}\sum_{k=0}^{Nt} \pr({\kappa(\intpart{Nt})=k})^{2/3}\\
& \le \E[\tau_{2}^3]^{1/3} (Nt)^{1/3}. 
\end{align*}
Therefore, 
   \begin{align*}
 & \mathbb{P}\left(\mathbf{d}\left(\delta_{N^{-1/2}}\left(\Embedding^{(1)}(X)_{T_{\kappa(\lfloor Nt \rfloor)}}\right),
 \delta_{N^{-1/2}}\left(\Embedding^{(1)}(X)_{\intpart{Nt}}\right)\right) >\epsilon\right)
      \leq
 \dfrac{Kd(\E[\tau_{2}^3]t)^{1/3}}{N^{1/6}\epsilon}\to 0 \text{ as }N\to\infty.
     \end{align*}
Next, using the strong law of large numbers together with the decomposition of $\tau_{k}=\sum_{\ell=1}^k (\tau_{\ell}-\tau_{\ell-1})$ into independent variables, with the same distribution for $\ell>1$, one deduces that
$\kappa(n)/n$ converges a.s.\ to $\beta^{-1}$.
Hence the conclusion of Step $2$ together with Slutsky's Theorem \cite{slutsky1923stochastische}
imply the convergence in distribution
   \begin{align*}
   \Embedding^{(N)}(X) \to (B_t,\mathcal{S}_t + t\Gamma )
   \end{align*}
for any fixed $t\in [0,T]$.
Extending the convergence to all finite-dimensional marginals of $\Embedding^{(N)}(X)$
is done similarly using Slutsky's Theorem on $\reals^d\times \reals^{d\otimes d}$.
%%%%% fourth step

\textbf{Step 4: }
It is left to prove the tightness of the process. In order to do this, we use the
Kolmogorov tightness criterion for rough paths \cite[Theorem 3.10]{friz2014course}.
That is, in order to obtain tightness for $\alpha<\dfrac{p^*-1}{2p^*}$.
it is enough to show that there exists a positive constant $c$ such that, for all $0\leq s<t\leq T$,
%\begin{align*}
%\sup_N \E\left[\mathbf{d}\left(\Embedding^{(N)}( X)_s,\Embedding^{(N)}( X)_t \right)^{2p^*}\right] =
% O ( |t-s|^{p^*}).
%\end{align*}
%This mean that we need to show
\begin{equation}\label{eq: Kolmog condition}
\sup_N \E\left[\norm {\Embedding^{(N)}( X)_{s,t}}^{2p^*}\right] \le
 c  |t-s|^{p^*}.
\end{equation}
To avoid heavy notation we write $\mathbf{X}_{s,t}:=\Embedding^{(1)}(X)_{s,t}$
and assume without loss of generality that $\tau_1$ has the same distribution as $\tau_k - \tau_{k-1}$ for $k>1$.
From the definition of iterated integral and the fact the paths are linear interpolations of discrete paths
proving \eqref{eq: Kolmog condition} boils down to showing that there is a constant $c$ so that
\[
  \E\left[\norm{\mathbf{X}_{\ell,k}}^{2p^*}\right] \le c (k-\ell)^{p^*}
\]
uniformly on $0\le \ell<k\le NT$.
Note that by the i.i.d\ regeneration structure
\[
  \E\left[\norm{\mathbf{X}_{\tau_\ell,\tau_k}}^{2p^*}\right] =   \E\left[\norm{\mathbf{X}_{\tau_k-\tau_\ell}}^{2p^*}\right].
\]
The tightness argument \cite[Step 2 in Chapter 3]{breuillard2009random} then immediately implies
%\[
%  \E\left[\norm{\mathbf{X}_{\tau_k}}^{2p^*}] = O(\tau_k^{p^*}),
%\]
%which in turn implies
\[
  \E\left[\norm{\mathbf{X}_{\tau_k}}^{2p^*}\right] = O(k^{p^*}),
\]
where we used the fact that $\E[\tau_k^{p^*}] = O ( k^{p^*})$.
Next, if $k,\ell$ are in the same regeneration interval, the fact that the jumps are bounded,
regeneration intervals have finite $2p^*$ moments, and the definition \eqref{def:CarnotCaraDistance} imply
\[
  \E\left[\norm{\mathbf{X}_{\ell,k}}^{2p^*} \one_{\kappa(\ell)=\kappa(k)}\right] \le C_p'
\]
for some constant $C_p'$.
Therefore by sub-additivity \eqref{eq:CC norm subadditive},
and using H\"older's inequality together with \eqref{eq:CC norm bound} we can find a constant $C_{2p^*}$ so that
\[
  \E\left[\norm{\mathbf{X}_{\ell,k}}^{2p^*}\right]
    \le C_{2p^*}\Big(2C_p'  + \E\left[\norm{\mathbf{X}_{\tau_{\kappa(\ell)},\tau_{\kappa(k)+1}}}^{2p^*}\right]\Big)
    = O((k-\ell)^{p^*}).
\]

%First, by \cite[Step 2 in Chapter 3]{breuillard2009random} we have
%   \begin{align}
%      \label{eq:Kolmogorov}
%   \E[\norm{\mathbf{X}_{\tau_{\kappa(n)}}}^{2p^*}] \leq \underset{l=1,\ldots,n}{\sup} \E[\norm{\mathbf{X}_{l}}^{2p^*}]= O(n^{p^*}).
%  \end{align}
%We are left to bound $\mathbf{X}_{\tau_{\kappa(n)}}^{-1}\otimes \mathbf{X}_{n}$.
% We use \eqref{eq:bound} to get
%    \begin{align*}
%      \E{\norm{\mathbf{X}_{\tau_{\kappa(n)}}^{-1}\otimes \mathbf{X}_{n}}^{2p^*}}\leq \tilde R^{2p^*}\E[(\tau_2-\tau_1)^{2p^*}]
%    \end{align*}
%  The definition of $p^*$ implies that $2p^*\leq 2p$ and thus
%  $\E[(\tau_2-\tau_1)^{2p^*}]<\infty$ by the moment condition \eqref{eq:momentcondition}.
%  Consequently, we have that
%    \[
%     \E{\norm{\mathbf{X}_{\tau_{\kappa(n)}}^{-1}\otimes \mathbf{X}_{n}}^{2p^*}} = O(1).
%    \]
We conclude that the Kolmogorov criterion is satisfied and so the sequence $(\Embedding^{(N)}(X))_N$ is tight in $\mathcal{C}^{\alpha}\left([0,T],G^2(\mathbb{R}^d\right)$, $\alpha<\dfrac{p^*-1}{2p^*}$.
\end{proof}

%%%%%%%%%%%%%%%%%%%%%%%%%%%%%%%%%%%%%%%%%%%%%%%%%%%%%%%%%%%%%%%%5
\section{Acknowledgments}\label{sec:acknowledgements}
%%%%%%%%%%%%%%%%%%%%%%%%%%%%%%%%%%%%%%%%%%%%%%%%%%%%%%%%%%%%%%%%%5
The research of T.O.\ was funded by the German Research Foundation through the research unit
FOR 2402 – Rough paths, stochastic partial differential equations and related topics.
We would like to thank Damien Simon for a stimulating email exchange which contributed to the current presentation of the paper.
\bibliography{AreaAnomaly}
\bibliographystyle{alpha}
%plain,unsrt,alpha,abbrv,acm,apalike

%%%%%%%%%%%%%%%%%%%%%%%%%%%%%%%%%%%%%%%%%%%%%%%%%%%%%%%%%%%%%%%%%5
%\begin{appendices}
%%%%%%%%%%%%%%%%%%%%%%%%%%%%%%%%%%%%%%%%%%%%%%%%%%%%%%%%%%%%%%%%%%%%%%%%%%%%%%%%%%%%%%%%5
%\section{}\label{appendix:app1}

%\end{appendices}

%%%%%%%%%%%%%%%%%%%%%%%%%%%%%%%%%%%%%%%%%%%%%%%%%%%%%%%%%%%%%%%%%5
%%%%%%%%%%%%%%%%%%%%%%%%%%%%%%%%%%%%%%%%%%%%%%%%%%%%%%%%%%%%%%%%%5

\end{document}